\documentclass[11pt]{article}
\usepackage{amsmath}
\usepackage{amsfonts}
\usepackage{amssymb}
\usepackage{amsthm}
\usepackage{breqn}
\usepackage{setspace}
\usepackage{fullpage}
\usepackage{enumitem}
\usepackage{comment}
\usepackage{hyperref}
\usepackage{bbm}
\usepackage{tikz}
\usepackage{mathtools} 
\usetikzlibrary{patterns,arrows,decorations.pathreplacing}

\newtheorem{theorem}{Theorem} 
\newtheorem{corollary}{Corollary} 
\newtheorem{definition}{Definition}
\newtheorem{claim}{Claim}
\newtheorem{conjecture}{Conjecture}

\newtheorem{proposition}{Proposition} 
\newtheorem{observation}{Observation}

\newcommand{\F}{\mathcal{F}}

\newcommand{\abs}[1]{\left\lvert{#1}\right\rvert}
\newcommand{\floor}[1]{\left\lfloor{#1}\right\rfloor}
\newcommand{\ceil}[1]{\left\lceil{#1}\right\rceil}

\DeclareMathOperator{\ex}{ex}

\newcommand{\itn}{\ex^{-1}}

\title{Inverse Tur\'an numbers}
\begin{document}



 \author{
Ervin Gy\H{o}ri\thanks{Alfr\'ed R\'enyi Institute of Mathematics, Hungarian Academy of Sciences. email:~\texttt{gyori.ervin@renyi.mta.hu},~\texttt{salia.nika@renyi.hu}.} \footnotemark[2] \and
Nika Salia\footnotemark[1] \thanks{Central European University, Budapest.} \and 
Casey Tompkins\thanks{Discrete Mathematics Group, Institute for Basic Science (IBS), Daejeon, Republic of Korea. {email:~\texttt{ctompkins496@gmail.com}.}} \and 
Oscar Zamora\footnotemark[2] \thanks{Universidad de Costa Rica, San Jos\'e. email: \texttt{oscarz93@yahoo.es}.}
\and
}

\maketitle

\begin{abstract}


For given graphs $G$ and $F$, the Tur\'an number $\ex(G,F)$ is defined to be the maximum number of edges in an $F$-free subgraph of $G$. Foucaud, Krivelevich and Perarnau and later independently Briggs and Cox introduced a dual version of this problem wherein for a given number $k$, one maximizes the number of edges in a host graph $G$ for which $\ex(G,H) < k$.  

Addressing a problem of Briggs and Cox, we determine the asymptotic value of the inverse Tur\'an number of the paths of length $4$ and $5$ and provide an improved lower bound for all paths of even length. Moreover, we obtain bounds on the inverse Tur\'an number of even cycles giving improved bounds on the leading coefficient in the case of $C_4$. Finally, we give multiple conjectures concerning the asymptotic value of the inverse Tur\'an number of $C_4$ and $P_{\ell}$, suggesting that in the latter problem the asymptotic behavior depends heavily on the parity of $\ell$.

\end{abstract}

\section{Introduction}
Tur\'an's theorem~\cite{turan} asserts that the maximum number of edges in a subgraph of the complete graph $K_n$ on $n$ vertices with no subgraph isomorphic to the complete graph on $r$ vertices is attained by the complete $r$-partite graph with parts of size $\floor{n/r}$ and  $\ceil{n/r}$. This graph is referred to as the Tur\'an graph and is denoted by $T(n,r)$.

Since Tur\'an's seminal result, the problem of maximizing the number of edges in an $n$-vertex graph not containing a fixed graph $F$ as a subgraph has been investigated for a variety of graphs $F$. A graph $G$ containing no member of $\F$ as a subgraph is said to be $\F$-free, and for $\F=\{F\}$ we say that such a graph is $F$-free. The Tur\'an number $\ex(n,\F)$ is defined to be the maximum number of edges in an $\F$-free subgraph of $K_n$.  The classical Tur\'an problem was settled asymptotically for all finite families of graphs $\F$ of chromatic number at least three by Erd\H{o}s, Stone and Simonovits~\cite{erdosstonesimonovits,ess}.  However, for most bipartite graphs $F$, the Tur\'an problem remains open. 

More generally for a given host graph $G$, the Tur\'an number $\ex(G,\F)$ is defined to be the maximum number of edges in an $\F$-free subgraph of $G$ (so $\ex(n,\F) = \ex(K_n,\F)$).  Common alternative host graphs include the complete bipartite graph $K_{m,n}$ (the so-called Zarankiewicz problem), the hypercube $Q_n$~\cite{cube}, a random graph~\cite{random}, as well as the class of $n$-vertex planar graphs~\cite{Dowden}.  

 In this paper, we will be concerned with a dual version of Tur\'an's extremal function introduced by Foucaud, Krivelevich and Perarnau~\cite{Foucaud} and later (in a different but equivalent form which we will use) by Briggs and Cox~\cite{cox}.  The number of vertices and edges in a graph $G$ are denoted by $v(G)$ and $e(G)$, respectively. The \emph{inverse Tur\'an number} is defined as follows.      

\begin{definition}
For a given family of graphs $\F$,
\[
\itn(k,\F) = \sup \{e(G): \mbox{$G$ is a graph with $\ex(G,\F) < k$} \}.
\]
For $\F=\{F\}$, we write $\itn(k,\{F\}) = \itn(k,F)$.
\end{definition}
Note that $\itn(k,\F)$ may be infinite. However, Briggs and Cox~\cite{cox} observed that $\itn(k,F)$ is finite whenever $F$ is not a matching or a star.  
An equivalent formulation of the problem is that we must find the maximum number of edges in a graph $G$ such that any subgraph of $G$ with $k$ edges contains a copy of some $F \in \F$. Observe that if $F_1$ is a subgraph of $F_2$, then $\itn(k,F_1) \ge \itn(k,F_2)$. Throughout this paper, when discussing inverse Tur\'an numbers, the asymptotic notation $O$ and $\Omega$ indicates that $k$ tends to infinity, and constants involving other parameters may be hidden.

Briggs and Cox~\cite{cox} gave upper and lower bounds on the inverse Tur\'an number of $C_4$ of the form $\Omega(k^{4/3})$ and $O(k^{3/2})$, respectively.  Unknown to Briggs and Cox at the time, this problem was considered earlier in a different form by Foucaud, Krivelevich and Perarnau~\cite{Foucaud} where a bound was proved that was sharp up to a logarithmic factor. Even more, according to Perarnau and Reed~\cite{reed} the problem was already proposed by Bollob\'as and Erd\H{o}s at a workshop in 1966 (see~\cite{komlos} for a related problem about union-free families from 1970).  More generally a recursive bound on the inverse Tur\'an number of $\itn(k,\{C_4,C_6,\dots,C_{2t}\})$ was also obtained in~\cite{Foucaud}, which is also tight up to a logarithmic factor.  

For graphs $F$ with chromatic number at least $3$, Foucaud, Krivelevich and Perarnau~\cite{Foucaud} and Briggs and Cox~\cite{cox} determined the inverse Tur\'an number asymptotically. Moreover, Briggs and Cox~\cite{cox} determined the inverse Tur\'an number of the complete graph precisely as well as the union of a path of length $1$ and a path of length $2$. They also settled the case of paths of length $3$ and proposed a conjecture about the inverse Tur\'an number of a path of length $4$. 

  In Section~\ref{secpath} we will investigate the inverse Tur\'an problem for paths, resolving a conjecture of Briggs and Cox asymptotically and providing a new lower bound for paths of any even length.  In Section~\ref{cycbip} we will determine the order of magnitude of the inverse Tur\'an number of any complete bipartite graph resolving another conjecture of Briggs and Cox about the order of magnitude of $\itn(k,C_4)$.  We note however, that this conjecture already follows directly from an unpublished preprint of Conlon, Fox and Sudakov~\cite{conlon} which preceded the paper of Briggs and Cox~\cite{cox}, but we provide a proof in the formulation introduced by Briggs and Cox for completeness. In the case of $C_4$, we give improved bounds on the leading coefficient and conjecture that the lower bound is optimal. Additionally, we give some estimates on the inverse Tur\'an number of an arbitrary even cycle.  Finally in Section~\ref{conjectures} we present some conjectures and directions for future work. 

We conclude this section by introducing some notation which we will require in our proofs. For any graph $G$, let $V(G)$ and $E(G)$ denote the set of vertices and edges of $G$, respectively. For a graph $G$ with bipartition $(X,Y)$, the number of edges between the vertex sets $X$ and $Y$ is denoted by $e(X,Y)$. For graphs $G$ and $H$, the number of copies of $H$ in $G$ is denoted by $\mathcal{N}(H,G)$.  The path and cycle with $t$ edges are denoted by $P_t$ and $C_t$, respectively.  For a graph $G$ and a set $X \subseteq V(G)$, the induced subgraph of $G$ on $X$ is denoted $G[X]$.  For a graph $G$ and subgraph $H$ of $G$, we denote by $G-H$ the induced subgraph of $G$ on $V(G)\setminus V(H)$.

\section{Inverse Tur\'an numbers of paths}\label{secpath}
In this section, we investigate the inverse Tur\'an problem for paths. We begin by recalling a well-known result of Erd\H{o}s and Gallai.
\begin{theorem}[Erd\H{o}s, Gallai~\cite{erdHos1959maximal}]
\label{eg}
For all $n \ge t$,
\begin{displaymath}
\ex(n,P_t) \le \frac{(t-1)n}{2},
\end{displaymath}
and equality holds if and only if $t$ divides $n$ and $G$ is the disjoint union of cliques of size $t$.
\end{theorem}
Observe that the extremal graphs given in Theorem~\ref{eg} are not connected in general. Kopylov~\cite{Kopylov} determined the extremal number of paths under the additional assumption that the graph is connected. Balister, Gy{\H{o}}ri, Lehel and Schelp~\cite{BGLS} strengthened Kopylov's result by characterizing the extremal graphs for all $t$ and $n$. For simplicity we will state a simple consequence of Kopylov's result which we will require: For sufficiently large $n$, the maximum number of edges in an $n$-vertex, connected graph with no path of length $t$ is $\floor{\frac{t-1}{2}}n+f(t)$, for some fixed function $f(t)$. 

Surprisingly the asymptotic upper bounds on the number of edges in the connected and general cases are the same when $t$ is odd, but if $t$ is even the coefficient of $n$ is $\frac{t-2}{2}$ in the connected case instead of $\frac{t-1}{2}$ as in the general case. 
For this reason, we obtain different lower bounds for the inverse Tur\'an numbers of paths, depending on the parity of their lengths.

 \begin{theorem}[Briggs, Cox~\cite{cox}]\label{Lower_bound_Cox}
For all $t \ge 3$,
 \[\ex^{-1}(k,P_{t}) \geq \binom{\floor{\frac{2k}{t-1}}-1}{2}.\]
  \end{theorem}
  
The bound in Theorem~\ref{Lower_bound_Cox} comes from taking a complete graph of the appropriate size and applying Theorem~\ref{eg}. 
In the case of $t=3$, Briggs and Cox~\cite{cox} proved that a complete graph gives the optimal bound for $\ex^{-1}(k,P_{3})$. 
Briggs and Cox also noted that for $P_4$ one can do better by considering a complete bipartite base graph and using a result of Gy\'arf\'as, Rousseau and Schelp~\cite{gyarfas} on the extremal number of $P_t$ in such graphs.  However, starting with a clique is superior to a complete bipartite graph for $P_t$, $t \neq 4$.  We will improve the lower bound on $\ex^{-1}(k,P_{2t})$ in general by considering balanced complete multipartite graphs. Note that since the inverse Tur\'an number is non-decreasing when considering supergraphs, it follows that the inverse Tur\'an number of any path of length at least~$3$ is $\Theta(k^2)$.

We will make use of the following celebrated result of Dirac.
 \begin{theorem}[Dirac~\cite{dirac}]
\label{delta}
Let $H$ be a connected graph with minimum degree $\delta(H) \ge t$, then $v(H) \le 2t$ or $H$ contains $P_{2t}$.

\end{theorem}

 In our proofs we will also require the following famous result of Erd\H{o}s, R\'enyi and  S\'os.

\begin{theorem}[Erd\H{o}s, R\'enyi, S\'os \cite{ErdRenyiSos}]\label{c4free}
\[\ex(n,C_4)\leq \frac{1}{2} n^{\frac{3}{2}}+\frac{1}{2}n. \]
\end{theorem}

We use the following proposition to provide a new lower bound on the inverse Tur\'an number for paths of even length by taking Tur\'an graphs as base graphs.
 
 \begin{proposition} \label{pr}
 
 \[\ex(T(n,r),P_{2t}) = \begin{cases} n(t-1) + O(1)&\mbox{ if } 2 \le r \le t,\\ n\min\{\frac{2t-1}{2},\frac{2t-3}{2}+\frac{r}{2t}\} + O(1)&\mbox { if } r>t.\end{cases}\] 
 \end{proposition}

   \begin{proof}
  Let $G$ be the graph $T(n,r)$, and let $H$ be a $P_{2t}$-free subgraph of $G$ with the maximum number of edges. 
  Repeatedly remove the vertices of $H$ of degree less than $t$ until a graph $H'$ of minimum degree at least $t$ is obtained. 
  Let $C_1,\dots,C_s$ be the components of $H'$.
  Since $\delta(C_i) \geq t$, 
  by Theorem~\ref{delta} we have that $v(C_i)\leq 2t$.
 If $v(C_i) \leq 2t-1,$ we would have that   
  $e(C_i) \leq \binom{v(C_i)}{2} = v(C_i)\frac{(v(C_i)-1)}{2} \leq v(C_i)(t-1)$.
 
 If $v(C_i) = 2t$, then $e(C_i) \leq e(T(2t,r))$, for $t\leq r \leq 2t,$ we have that \begin{displaymath}
 e(T(2t,r)) \begin{cases}
= t(2t-1) -(2t-r) = 2t\left(\frac{2t-3}{2}+\frac{r}{2t}\right) = v(C_i)\left(\frac{2t-3}{2}+\frac{r}{2t}\right) & \mbox{ for } t\leq r \leq 2t, \\
= t(2t-1) = v(C_i)\left(\frac{2t-1}{2}\right) & \mbox{ for } 2t\leq r,\\
\leq t(2t-1) - t = 2t(t-1) =  v(C_i)(t-1) & \mbox{ for } 2 \le r \leq t. \\
 \end{cases}
 \end{displaymath}
 
   Summing over the components and the vertices which were removed, we obtain \begin{displaymath}
      e(H) \leq \begin{cases}     v(H)(\frac{2t-3}{2}+\frac{r}{2t}) \leq n(\frac{2t-3}{2}+\frac{r}{2t}) & \mbox{ if }  t < r < 2t,\\     v(H)(t-1) \leq n(t-1) & \mbox{ if }  2t \leq r, \\ v(H)(t-1) \leq n(t-1) & \mbox{ if }  2 \leq r \leq t.
      \end{cases}
      \end{displaymath}
      
For a matching lower bound in Proposition~\ref{pr}, we will consider certain subgraphs of $G$.  If $2\leq r\leq t,$ let $A$ be a color class of $G$ and let  $\{v_1,v_2,\dots,v_{t-1}\}$ and $\{u_1,u_2,\dots,u_{t-1}\}$ be two sets of $t-1$ vertices in $A$ and $V(G) \setminus A$, respectively. 
Construct a graph $H$ by adding an edge from $v_i$ to every vertex in $V\setminus (A \cup \{u_1,u_2,\dots, u_{t-1}\}$ and from $u_i$ to every vertex in $A \setminus \{v_1,v_2,\dots,v_{t-1}\}$ for each $i=1,2,\dots,t$.  Then $e(H) = (t-1)(n-2t+2) = n(t-1)+O(1)$ and $H$ is $P_{2t}$-free, since $H$ is the disjoint union of graphs of the form $K_{t-1,m}$.

If $r\geq t+1$, then let $w_1,w_2,\dots,w_n$ be a numbering of the vertices of $G$ such that for $i=1,2,\dots r$, the sets $A_i = \{w_j \in V: j\equiv i \pmod r\}$ are the color classes of $G$. 
Then we can take $\floor{n/2t}$ copies of $T(2t,r)$ by taking $H$ as the disjoint union of the graphs $G[\{w_{2tm+1},w_{2tm+2},\dots,w_{2t(m+1)}\}],$ for $m=1,2, \dots, 2t$. It follows that $H$ is a $P_{2t}$-free graph with $e(H) =  n(\frac{2t-3}{2}+\frac{r}{2t}) + O(1)$, for $t<r<2t$ and $e(H) = n(\frac{2t-1}{2}) + O(1)$ for $r\geq 2t.$
\end{proof}


 
 


\begin{corollary}
Among the Tur\'an graphs $T(n,r)$ with $\ex(T(n,r),P_{2t})<k$, the one with the maximum number of edges is obtained by $r =t$ and $n =  \floor{\frac{k-1}{t-1}} + O(k)$. 
In particular, for $t \ge 2$,
 \[\ex^{-1}(k,P_{2t}) \geq e\left(T\left(\floor{\frac{k-1}{t-1}},t\right)\right) =
 \frac{(k-1)^2}{2t(t-1)} + O(k).\]
\end{corollary}

\begin{proof}

Take $n$ and $r$ such that $\ex(T(n,r),P_{2t}) = k - 1.$ 

If $2\leq r\leq t$, then $k = n(t-1) + O(1)$, so $n=\frac{k}{t-1} + O(1)$, and \begin{displaymath}
e(T(n,r)) = \floor{\frac{r-1}{2r}n^2} \leq \floor{\frac{t-1}{2t}n^2} = \frac{k^2}{2(t-1)t} + O(k),
\end{displaymath}
the maximum is only achieved when $t=r$.

If $t<r<2t$, then $k = n(\frac{2t-3}{2}+\frac{r}{2t}) + O(1)$, so $n = \dfrac{k}{\frac{2t-3}{2}+\frac{r}{2t}} + O(1),$ and
\begin{displaymath}
e(T(n,r)) = \floor{\frac{r-1}{2r}n^2} = \frac{r-1}{2r} \left(\dfrac{k}{\frac{2t-3}{2}+\frac{r}{2t}}\right)^2 + O(k).
\end{displaymath}

Since $\dfrac{\partial}{\partial x}\left(\dfrac{x-1}{2x} \left(\dfrac{k}{\frac{2t-3}{2}+\frac{x}{2t}}\right)^2\right) = \dfrac{2t^2 k^2(t-x)(2t+2x-3)}{x^2(t(2t-3)+x)^3}$ is negative for $t<x< 2t$ (because $t\geq 2$), it follows that  \begin{displaymath}
\frac{r-1}{2r} \left(\dfrac{k}{\frac{2t-3}{2}+\frac{r}{2t}}\right)^2 + O(k) \le \frac{k^2}{2t(t-1)} +O(k).
\end{displaymath}

If $r\geq 2t$, then $k= \frac{2t-1}{2}n + O(1),$ so $n = \frac{2k}{2t-1} + O(1)$ and consequently
\begin{displaymath}
e(T(n,r)) \leq \binom{n}{2} \leq  \frac{2k^2}{(2t-1)^2} + O(k) \le \frac{k^2}{2t(t-1)}+O(k),
\end{displaymath}
since $\frac{2}{(2t-1)^2} < \frac{1}{2t(t-1)}$.
\end{proof}


 \begin{theorem}\label{P4} 
 We have $\itn(k,P_4) = k^2/4 +O(k^{3/2})$. 
\end{theorem}
\begin{proof}

Let $G$ be a graph with no $P_4$-free subgraph of $k$ edges and no isolated vertices. Clearly, $G$ does not contain a star forest of $k$ edges as it would be $P_4$-free, and so we have $v(G)<2k$. 

We next partition the vertex set of $G$ into four parts. In the definition of the partition let $x$, $y$ and $z$ be non-negative integers such that $x$ is maximal, and subject to $x$ being maximal $y$ is maximal, and subject to $y$ being maximal $z$ is maximal. Take $x$ pairwise vertex-disjoint copies of $K_4$ in $G$, and let $X$ be the vertex set of their union.  Take $y$ pairwise vertex-disjoint copies of $K_4^-$ ($K_4$ with one edge deleted) in $G-X$, and let $Y$ be the vertex set of their union.  Finally, take $z$ pairwise vertex-disjoint copies of $C_4$ in $(G-X)-Y$, and let $Z$ be the vertex set of their union. Let $S=V(G)\setminus (X\cup Y \cup Z)$. Observe that $|X|=4x$, $|Y|=4y$ and $|Z|=4z$.

We are going to estimate the number of edges in $G$ by carefully analyzing the partition of the vertices. The number of edges in the set $X$ is at most number of edges in a complete graph on $4x$ vertices, so  $e(G[X])\leq \binom{4x}{2}$.  We have $e(G[Y])\leq 3(4y/3)^2$ by  Tur\'an's theorem since $G[Y]$ is $K_4$-free by the maximality of $X$. We have $e(G[Z])\leq 4z^2$ since $G[Z]$ is $K_4^-$-free by the maximality of $Y$ (the Tur\'an number of $K_4^-$ for all $n$ is a simple exercise and also follows from the results of Rademacher and Erd\H{o}s, see~\cite{olderdos}) . We have $e(G[S])=O(k^{3/2})$, since $G[S]$ is $C_4$-free by the maximality of $Z$ and  Theorem~\ref{c4free}.
    
From the maximality of $X$ we have $e(X,Y) \leq 14xy$. Indeed, otherwise there would be a $K_4$ in $X$ and $K_4^-$ in $Y$, connected with at least $15$ edges. In this case it is easy to find two vertex-disjoint copies of $K_4$ among the $8$ vertices, contradicting the maximality of $X$.

From the maximality of $X$ and $Y$, we have $e(X,Z) \leq 12xz$. Suppose otherwise, then there will be a $K_4$ in $X$ and $C_4$ in $Z$, connected with at least $13$ edges. For the $4$ vertices in the $C_4$, let us consider the vector indicating the number of edges from each vertex of the $C_4$ to the $K_4$. The possibilities are $(4,4,4,1)$, $(4,3,3,3)$ or $(4,4,3,2)$, up to reordering. 
     
     In the case $(4,4,4,1)$, the vertex from the $C_4$ with one neighbor in the $K_4$ along with its neighbors in the $C_4$ and neighbor in the $K_4$, yields a $K_4^-$.  The remaining vertex from the $C_4$ and the remaining $3$ vertices from the $K_4$ yield a $K_4$, contradicting the maximality of $X$ and $Y$. Now consider the case $(4,3,3,3)$. Take two adjacent neighbors from the $C_4$ each having $3$ neighbors in the $K_4$.  These two vertices have at least two common neighbors in the $K_4$, take these as well to obtain a $K_4$. The remaining $4$ vertices yield a copy of $K_4^-$, contradicting the maximality of~$X$ and~$Y$.

     Finally, consider the case $(4,4,3,2)$. If the vertices of the $C_4$ with $3$ or $2$ neighbors in the $K_4$ are adjacent in the $C_4$ or share two common neighbors from the $K_4$, then we have a vertex-disjoint $K_4$ and $K_4^-$ (or two disjoint copies of $K_4$) on these eight vertices, a contradiction to the maximality of  $X$ and $Y$. Otherwise, a vertex from the $C_4$ with $4$ neighbors in the $K_4$ along with the vertex in the $C_4$ with $2$ neighbors in the $K_4$ and its two neighbors yields a $K_4$. Moreover, the remaining $4$ vertices also induce a $K_4$, contradicting the maximality of $X$.

        
    We now use the maximality of $Y$ and $Z$, to show $e(Y,Z) \leq 10yz$. Otherwise there will be a $K_4^-$ in $Y$ and $C_4$ in $Z$, connected with at least $11$ edges. First observe that there is no vertex of the $C_4$ adjacent to three vertices of $K_4^-$ which form a triangle, from the maximality of $X$. Hence there is no vertex from the $C_4$ adjacent to all vertices of the $K_4^-$. Even more, if there is a vertex from the $C_4$ adjacent to three vertices in the $K_4^-$, then its neighborhood must contain the two nonadjacent vertices of $K_4^-$. Also if there are two adjacent vertices in the $C_4$ each adjacent to three vertices in the $K_4^-$, then they must be adjacent to different triples for otherwise we would have a $K_4$. 
    Since there are at least $11$ edges between the $K_4^-$ and the $C_4$, three of the vertices of the $C_4$ have three neighbors in the $K_4^-$ and the forth has at least two. Let the vertices of $K_4^-$ be $v_1,v_2,v_3,v_4$ with the missing edge $(v_2,v_4)$, and let the vertices of the $C_4$ be $w_1,w_2,w_3,w_4$ in that order. Then from the observations given in this paragraph and without loss of generality, we may assume that the neighbors of $w_1$ and $w_3$ are $v_1,v_2,v_4$ and the neighbors of $w_2$ are $v_2,v_3,v_4$. The neighbors of $w_4$ are either $\{v_2,v_4\}$ or $\{v_3,v_i\}$ for some $i \neq 3$. If the neighbors of  $w_4$ are  $\{v_2,v_4\}$, then $v_2,w_1,w_3,w_4$ induce a $K_4^-$ and the remaining vertices also induce a $K_4^-$, a contradiction to to the maximality of~$Y$. Otherwise if the neighbors of $w_4$ are $v_3$ and some other vertex, then $w_4$ with both  neighbors in the $K_4^-$ and $w_3$  induce  a $K_4^-$ or $K_4$ as do the rest of the vertices, contradicting the maximality of $X$ and $Y$.
    
    \begin{observation} \label{Observation_Number of incidences}
   Any selected $K_4$, $K_4^-$ or $C_4$ block has at most $k':=k-6x-5y-4z+6$ neighbors in the set $S$. Otherwise for each vertex of $S$ incident to the chosen block we may choose a neighbor from the block. In this way, we find at most four vertex-disjoint stars with at least $k'$ edges. These stars along with the rest of the remaining blocks contain at least $k$ edges and no $P_4$, a contradiction to our initial assumption.
    \end{observation}

     From Observation~\ref{Observation_Number of incidences}, we have $e(X,S) \leq 4xk'$. We also have $e(Y,S) \leq 3yk'$, since by the maximality of $X$, there is no vertex of $S$ incident to all vertices of a chosen $K_4^-$. Similarly, by the maximality of $Y$, no vertex in $S$ is adjacent to at least three vertices in a selected $C_4$ and so $e(Z,S) \leq 2zk'$.

Finally applying all these estimates, we obtain that $e(G)$ is at most
\begin{align*}
 &e(G[X]) + e(G[Y]) + e(G[Z]) + e(G[S])+e(X,Y) + e(X,Z) + e(Y,Z) +e(X,S) + e(Y,S) + e(Z,S) \\ &\leq
8x^2 + \frac{16}{3}y^2 + 4z^2 + O(k^{3/2}) + 14xy + 12xz +10yz + k'(4x+3y+2z)\\ &\leq
(k-4x-3y-2z)(4x+3y+2z)+ O(k^{3/2}) \\  &\leq k^2/4 +O(k^{3/2}). \qedhere
\end{align*}
 \end{proof}

\begin{theorem}
\[
    \ex^{-1}(k,P_5)=\frac{k^2}{8}+O(k).
\]
\end{theorem}

\begin{proof}
The lower bound can be obtained by considering the complete graph $K_{\floor{\frac{k-1}{2}}}$, see Theorem~\ref{Lower_bound_Cox}. 

To establish the upper bound we show that for all graphs $G$ with $e(G)\geq \frac{k^2}{8}+100 k$, there exists a subgraph of $G$ with $k$ edges and no copy of $P_5$. The proof is by induction on $k$. The base case is trivial. To establish the induction step, we divide the proof of the upper bound into cases depending on what substructures of the graph are present. First we show that if $G$ contains a $5$-vertex subgraph with at least $8$ edges, then we are done by induction. 

\begin{claim} \label{clcl}
Let $H$ be a $5$-vertex subgraph of $G$ with $e(H) = t$ such that $t$ is maximal. If $t=8, 9$ or $10$, then $G$ contains a $P_5$-free subgraph with $k$ edges.
\end{claim}
\begin{proof}
Let us fix such a $5$-vertex subgraph $H$. 
If $G - H$  contains a $P_5$-free subgraph with $k-t$ edges, then we are done. 
Otherwise, by induction we have $e(G -  H)\leq \frac{(k-t)^2}{8}+100(k-t)$. Let $E_2$ be the set of crossing edges from $G$, that is
\[
E_2=\Big\{(u,v) : u \in V(H) \mbox{ and } v \in V(G -  H)\Big\}.
\]
Let us denote the number of edges in $E_2$ by $e_2$ and the vertex set of  $G- H$ incident to $H$ by $V_2$. 
Then, since $t\leq 10,$  we have
\[
e_2\geq e(G)-t-\frac{(k-t)^2}{8}-100(k-t)\geq \frac{k^2}{8} +100k -t-\frac{(k-t)^2}{8}-100(k-t)>\frac{t}{4}k.
\]
Now we will show that, there is a $k$-edge, $P_5$-free subgraph of $G$ only containing edges from $E_2$. 
We note that for any $s$, $K_{2,s}$ is $P_5$-free.  For a fixed partition $\mathcal{P}=(A_1,A_2,A_3)$ of $V(H)$ into three sets of sizes $2$, $2$ and $1$, let us define a subgraph $G_H(\mathcal{P})$ of $G$ on the vertex set $V(H) \cup V_2$ and an edge set which is a subset of $E_2$ in the following way.
We assign each vertex $v \in V_2$ to a partition class $A_i$ to which there are the maximum number of neighbors of $v$,  and let those edges be in~$G_H(\mathcal{P})$. 
Therefore we take at least one and at most two edges for each vertex of $V_2$. 
Let us denote number of vertices from $V_2$ incident to $i$ vertices of $H$ (in $G$) by $n_i$, $i\in \{1,2,3,4,5\}$. 
Then we have $\sum_{i=1}^5 in_i=e_2> \frac{t}{4}k$. Let us take such a $(2,2,1)$-partition $\mathcal{P}$ uniformly at random.
Observing that there are $15$ such $(2,2,1)$-partitions, the expected number of edges in $G_H(\mathcal{P})$ is
\[
\mathbb{E}(e(G_H(\mathcal{P})))= n_1+\left(\frac{12}{15}\cdot1+\frac{3}{15}\cdot2\right)n_2+\left(\frac{6}{15}\cdot1+\frac{9
}{15}\cdot2\right)n_3+2n_4+2n_5\geq\frac{e_2}{2.5}.
\]
If $t = 10$, then $e_2 > 2.5k$ so $\mathbb{E}(e(G_H(\mathcal{P}))) > k,$ and we would be done, since there exists a $P_5$-free subgraph of $G$.

If $t= 8$ or $9$, note that $n_5 = 0$ otherwise we could replace the vertex of lowest degree of $H$ with a vertex adjacent to the five vertices of $H$ and obtain a subgraph of $G$ with 5 vertices and more than $t$ edges, so in this case 
\[
\mathbb{E}(e(G_H(\mathcal{P})))= n_1+\left(\frac{12}{15}\cdot1+\frac{3}{15}\cdot2\right)n_2+\left(\frac{6}{15}\cdot1+\frac{9
}{15}\cdot2\right)n_3+2n_4\geq\frac{e_2}{2}.
\]
Since $t \geq 8$, we have that $e_2 > \frac{t k}{4} \geq 2k$, thus  $\mathbb{E}(e(G_H(\mathcal{P}))) > k$, and we are done.
\end{proof}

Let us denote a copy of $K_5$ with one edge deleted by $K_5^{-}$ and any copy of $K_5$ with two missing edges (in either of the two non-isomorphic ways) by $K_5^{--}$. From Claim~\ref{clcl}, we may assume $G$ is $K_5^{--}$-free. Next we consider the case when $G$ contains a copy of $K_4$.
\begin{claim}
If $G$ contains $K_4$ as a subgraph, then $G$ contains a $k$-edge, $P_5$-free subgraph.
\end{claim}
\begin{proof}
Fix a subgraph $H$ isomorphic to $K_4$ and define $n_i$, $1 \le i \le 4$ and $e_2$, the number of crossing edges, as before. By a similar inductive argument as in the proof of Claim~\ref{clcl} we deduce that $e_2>3k/2$.  Since $G$ is $K_5^{--}$-free it follows that $n_4,n_3,n_2=0$, and so we have 
\[
e_2=n_1>k.
\]
Since $n_4,n_3,n_2=0$ there is clearly no $P_5$ among the edges in $E_2$, and we find a $P_5$-free subgraph of $G$ with at least $k$ edges.
\end{proof}

Now we may assume $G$ is $K_5^{--}$ and $K_4$-free. Next we consider the graphs obtained from the complete graph on $5$ vertices by deleting three edges. Notice that $K_5$  missing a star on four vertices contains $K_4$. So it remains to consider $K_5$ missing a triangle, $K_5$ missing a path of length~$3$ and $K_5$ missing an edge and a path of length~$2$, denoted by $K_5^{-K_3}$, $K_5^{-P_3}$ and  $K_5^{-P_1\cup P_2}$,  respectively. 
\begin{claim}
If $G$ contains $K_5^{-K_3}$ or $K_5^{-\{P_1\cup P_2\}}$  as a subgraph then $G$ contains a $k$-edge, $P_5$-free subgraph.
\end{claim}
\begin{proof}
In the same way as in the preceding claims we deduce that $e_2>\frac{7}{4}k+30$. Let $H$ be $K_5^{-K_3}$ or $K_5^{-P_1\cup P_2}$, then since $G$ is $K_5^{--}$-free, we have $n_5,n_4=0$. Moreover, there are no two vertices of degree three (in $H$) sharing the same neighborhood which induces more than one edge in $H$. Therefore by removing at most $\binom{5}{3}$ vertices of degree three it follows that if there are still degree-three vertices, then these vertices have the same neighborhood and this neighborhood induces the missing triangle for $K_5^{-K_3}$ or missing path of length~$2$ of $K_5^{-P_1\cup P_2}$. Finally we have $n_1+2n_2+3n_3> \frac{7}{4}k$ and all degree-three vertices share the same neighborhood. Let us denote the vertex-neighborhood of the degree-three vertices by $A:=\{v_1,v_2,v_3\}$.

In this case we search for a $(2,2,1)$-partition of the vertices $\{v_1,v_2,v_3,v_4,v_5\}$ by analyzing the neighborhoods. We consider partitions where one partition class of size two is a subset of $A$. For each such partition we may take two edges for each degree-three vertex. Consider the following seven classes:
$\bigg\{ \{v_1,v_2\}, \{v_3,v_4\}\bigg\},$
$\bigg\{ \{v_1,v_2\}, \{v_3,v_5\}\bigg\},$ 
$\bigg\{ \{v_1,v_2\}, \{v_4,v_5\}\bigg\}.$ 
$\bigg\{ \{v_2,v_3\}, \{v_1,v_4\}\bigg\},$
$\bigg\{ \{v_2,v_3\}, \{v_1,v_5\}\bigg\},$
$\bigg\{ \{v_1,v_3\}, \{v_2,v_4\}\bigg\} $ and
$\bigg\{ \{v_1,v_3\}, \{v_2,v_5\}\bigg\}.$ 
Note that their union contains every distinct pair of $\{v_i,v_j\}$, therefore since there are $n_2$ vertices incident to pairs of $\{v_i,v_j\}$, there are at least $n_2/7$ vertices of degree-two incident to pairs of vertices from the same class. Therefore fixing that partition $\mathcal{P}$, we get 

\[
e(G_H(\mathcal{P}) \ge n_1+\left(1+\frac{1}{7}\right)n_2+2n_3\ge \frac{4e_2}{7}>k.
\]
We are done since we have a $k$-edge, $P_5$-free subgraph of $G$.
\end{proof}

 From here we may assume $G$ is $K_5^{--}$, $K_5^{-K_3}$, $K_5^{-P_1\cup P_2}$ and $K_4$-free. 
 
\begin{claim}
If $G$ contains $K_5^{-P_3}$ as a subgraph, then $G$ contains a $k$-edge, $P_5$-free subgraph.
\end{claim}
\begin{proof}
Let $H$ be a copy of $K_5^{-P_3}$. In a similar way as before we can show that $e_2>\frac{7}{4}k+30$. Since $G$ is $K_5^{-P_1\cup P_2}$-free we have that $n_5, n_4=0$. Even more there are no two vertices of degree-three sharing the same neighborhood since $G$ is $K_5^{--}$, $K_5^{-K_3}$ and $K_5^{-P_1\cup P_2}$-free. Therefore  $n_3\leq \binom{5}{3}$, and we have $n_1+2n_2>\frac{7}{4}k$.  Taking a random $(2,2,1)$-partition $\mathcal{P}$ as in the earlier claims, we obtain

\[
\mathbb{E}(e(G_H(\mathcal{P})) \ge n_1+\left(\frac{12}{15}\cdot1+\frac{3}{15}\cdot2\right)n_2\geq\frac{3e_2}{5}>k.
\]
Therefore we are done since we have a $k$-edge, $P_5$-free subgraph of $G$.
\end{proof}

 From now on we may assume $G$ is $K_5^{---}$ and $K_4$-free, where $K_5^{---}$ denotes any of the graphs obtained from $K_5$ by deleting $3$ edges. 
\begin{claim}
If $G$ contains $K_4^-$ as a subgraph, then $G$ contains a $k$-edge, $P_5$-free subgraph.
\end{claim}
\begin{proof}
Let $H$ denote a copy of $K_4^-$ in $G$. Similarly to before we show that $e_2>\frac{5}{4}k$. Since $G$ is $K_5^{---}$-free, we have then $n_5,n_4,n_3,n_2=0$. Therefore we are done since we may take all edges from $E_2$ to obtain a $k$-edge, $P_5$-free subgraph of $G$. 
\end{proof}

 From here we may assume $G$ is $K_4^-$-free.
\begin{claim}
If $G$ contains $C_4$ as a subgraph, then $G$ contains a $k$-edge, $P_5$-free subgraph.
\end{claim}
\begin{proof}
Let $H$ denote a copy of $C_4$ in $G$.  Similarly to before we can show that $e_2>k+8$. Since $G$ is $K_4^{-}$-free we have then $n_5,n_4,n_3=0$. Even more there are no two vertices sharing the same neighborhood which induces an edge of $C_4$. Therefore dividing the vertices of $C_4$ into two classes of size two containing opposite vertices of the $C_4$, we obtain a subgraph of $G$ with at least $k$ edges and no $P_5$. 
\end{proof}

Finally since $G$ is $C_4$-free and the number of vertices of $G$ is at most $2k$ (as in Theorem~\ref{P4}), we obtain a contradiction to Theorem~\ref{c4free}.  
\end{proof}

\section{Inverse Tur\'an  number of complete bipartite graphs and even cycles.}\label{cycbip}

In order to get lower bounds for the inverse Tur\'an number of even cycles we will  apply the following theorem of Naor and Verstra\"ete.

\begin{theorem}[Naor, Verstra\"ete~\cite{NV}] \label{Cycle_Theorem}
For all $t \ge 2$,
\[ \ex(K_{n,m},C_{2t})\leq \begin{cases} (2t-3)((mn)^{\frac{t+1}{2t}}+m+n), & \mbox{if } t\mbox{ is odd,} \\ (2t-3)(m^{\frac{t+2}{2t}}n^{\frac{1}{2}}  +m+n), & \mbox{if } t\mbox{ is even.} \end{cases} \]
\end{theorem}

For the Tur\'an number of the complete bipartite graph $K_{s,t}$, K\H{o}v\'ari, S\'os and Tur\'an~\cite{kovari} proved the following theorem.

\begin{theorem}[K\H{o}v\'ari, S\'os, Tur\'an~\cite{kovari}]\label{kst} For all integers $s,t \ge 1$,
\[\ex(n,K_{s,t})  = O(n^{2-\frac{1}{s}}).\]
\end{theorem}
This result is only known to be of the correct order of magnitude in specific cases. In particular, Brown~\cite{BrownThom} determined the correct order of magnitude for $K_{3,3}$. Furthermore for $s>t!$ and later $s>(t-1)!$ the order of magnitude for $K_{s,t}$ was determined by Koll\'ar, R\'onyai and Szab\'o~\cite{krs} and Alon, R\'onyai and Szab\'o~\cite{ars}, respectively.  For more results on the classical Tur\'an number of bipartite graphs see the survey paper~\cite{FS}.

The corresponding inverse Tur\'an problem of determining the order of magnitude of $\itn(k,K_{s,t})$ for all $s$ and $t$ is a consequence of the results of Conlon, Fox and Sudakov~\cite{conlon}. Here, for completeness we give a proof using the formulation introduced by Briggs and Cox~\cite{cox}.

\begin{theorem}\label{bipartitr_inverse-turan}
Let $s,t$ be integers with $1<s\leq t$, then 
\[\itn(k,K_{s,t}) = \Theta(k^{1+\frac{1}{s}}).\]
\end{theorem}

\begin{proof}
For the lower bound, let $G$ be the complete bipartite graph with color classes $A$ and $B$ of sizes $k/s$ and $(k/t)^{1/s}$ respectively. Let $H$ be a subgraph of $G$ with $k$ edges. We will show that $H$ contains $K_{s,t}$ as a subgraph.
We have that the number of $(s+1)$-vertex stars in $H$ with center in $A$ is
\[\displaystyle \sum_{v\in A} \binom{d_H(v)}{s} \geq  \abs{A}\binom{\sum_{v\in A}d_H(v)/\abs{A}}{s} = \frac{k}{s}\binom{s}{s} =\frac{k}{s},\]
by Jensen's inequality. On the other hand, the number of $s$ element subsets of $B$ is \[\binom{(k/t)^{1/s}}{s} < \frac{k}{t s!} \leq \frac{k}{s t}.\]
Therefore there must exist an $s$-element subset $S$ of $B$ such that there are at least $t$ of the $(s+1)$-stars in $H$ with center in $A$ and having $S$ as the set of leaves. It follows that there is a copy of $K_{s,t}$ in $H$. 
Hence $\itn(k,K_{s,t}) \geq \frac{k^{1+1/s}}{st^{1/s}}$.

For the upper bound, we are going to use an approach introduced by Briggs and Cox~\cite{cox}. Observe that, it is enough to prove that $\itn(k,K_{s,s}) = O(k^{1+1/s}),$ since $\itn(k,K_{s,t}) \leq \itn(k,K_{s,s})$.
Let $G$ be a graph with $4k^{1+1/s}$ edges. Let $G_p$ be a randomly chosen subgraph obtained from $G$ by keeping each edge with probability $p= k^{-1/s}/2$, and let $H$ be the graph obtained from $G_p$ by removing an edge from every copy of $K_{s,s}$. 
Let $X$ be the number of edges of $G_p$ and $Y$ be the number of copies of $K_{s,s}$ in $G_p$. Then $\mathbb{E}[X] = 2k$ and \[\mathbb{E}[Y] = \frac{1}{2^{s^2}k^s}\mathcal{N}(K_{s,s},G) \leq  \frac{1}{2^{s^2}k^s}2^{2s}k^{1+s} \le k,\] where to bound $\mathcal{N}(K_{s,s},G)$ we use the simple estimate $\mathcal{N}(K_{s,s},G) \leq 2^s e(G)^s$ obtained from the fact that from any copy of $K_{s,s}$ we have an $s$-matching in $G$. It follows that $\mathbb{E}[X-Y]\geq k$, so there is exists a $K_{s,s}$-free subgraph $H$ of $G$ with at least $k$ edges.  
\end{proof}

In the case of $C_4$, we give a more precise calculation to prove upper and lower bounds within a factor of  $\frac{3\sqrt{3}}{2\sqrt{2}}<2$.

\begin{theorem}\label{c4bounds}
\[\floor{\sqrt{\frac{2}{3}k}} \floor{\frac{2}{3}k-1}\leq \itn(k,C_{4}) \leq k^{\frac{3}{2}} + o(k^{\frac{3}{2}}).\]
\end{theorem}

\begin{proof}
For the lower bound, we take a complete bipartite graph $G$ with color classes $A$ and $B$, with sizes $\floor{\sqrt{2k/3}}$ and $\floor{2k/3-1}$ respectively, and we show that every subgraph with $k$ edges contains a copy of $C_4$. Assume to the contrary that there is a subgraph $G'$ of $G$ which is $C_4$-free and has $k$ edges. Note that, without loss of generality, we may assume $G'$ has the same vertex set as $G$.  Therefore we have $\sum_{v \in B} d_{G'}(v)=k$. Let us count the number of cherries, that is copies of $K_{1,2}$, with the degree-two vertex from $B$ in the graph $G'$. On one hand, the number of such cherries is equal to  $\sum_{v \in B} \binom{d_{G'}(v)}{2}\geq \ceil{\frac{k}{3}}+1$, by convexity. On the other hand, for each pair of vertices from $A$, we may have at most one cherry, since $G'$ is $C_4$-free. Hence  number of such cherries is at most
$\binom{\floor{\sqrt{2k/3}}}{2}\leq k/3$, a contradiction. 
Therefore it follows that every subgraph of $G$ with $k$ edges contains a copy of $C_4$.   

The upper bound comes from a probabilistic argument similar to one given in the paper of Briggs and Cox~\cite{cox} in a more general setting. For the sake of completeness we provide the argument here. It is enough to show that for any graph $G$ with $e(G) \geq k^{3/2}+g(k)$, for some $g(k)=o(k^{3/2})$,  there exists a $C_4$-free subgraph of $G$ with $k$ edges. Without loss of generality, we may assume that $G$ is connected, since otherwise we can identify one arbitrary vertex from each component to form a connected graph with no $C_4$ using vertices from two components of the original graph, except for possibly the identified vertex. Since $G$ is connected there exists a spanning tree. It follows that if $v(G)>k$, then we have a cycle-free subgraph of $G$ with at least $k$ edges and we are done. Therefore we may assume $v(G)\leq k$. From Theorem~\ref{c4free}, there exists a $C_4$-free graph $H$, with the same vertex set as $G$, $v(H)=v(G)$ and $e(H) = \frac{1}{2} v(G)^{3/2}+o\left(v(G)^{3/2}\right)$. 

For any bijective function $f:V(G) \to  V(H)$, let us define a graph $F_f$ by $V(F_f):=V(G)$ and $E(F_f):=\{(u,v):(u,v) \in E(G) \mbox{ and } (f(u),f(v))\in E(H)  \}$. Observe that $F_f$ is a subgraph of $G$. By taking the bijection $f$ uniformly at random,
\begin{align*}
\displaystyle\mathbb{E}[e(F_f)]=\frac{2(v(G)-2)!}{v(G)!}e(G)e(H)&> \frac{v(G)^{\frac{3}{2}}+o\left(v(G)^{\frac{3}{2}}\right)}{v(G)^2}(k^{\frac{3}{2}} + g(k))\\ &\geq \left(v(G)^{-\frac{1}{2}}+o\left(v(G)^{-\frac{1}{2}}\right)\right)(k^{\frac{3}{2}} +g(k)) \\ &\geq (k^{-\frac{1}{2}}+o(k^{-\frac{1}{2}}))(k^{\frac{3}{2}} +g(k))\geq k,    
\end{align*}
for a suitably chosen $g(k)=o(k^{3/2})$. It follows that for some choice of $f$, we have a subgraph of $F_f$ and thus $G$ with at least $k$ edges and no $C_4$.  
\end{proof}

In the following theorem we offer some bounds for the inverse Tur\'an number of even cycles.


\begin{theorem}
Let $t \ge 2$, then
\[\itn(k,C_{2t})=\begin{cases}O(k^{2-\frac{2}{3t-3}}) & \mbox{if } t\mbox{ is odd,} \\ O(k^{2-\frac{2}{3t-2}}) & \mbox{if } t\mbox{ is even,} \end{cases} \]
and
\[\itn(k,C_{2t})=\begin{cases}\Omega(k^{2-\frac{2}{t+1}}) & \mbox{if } t\mbox{ is odd,} \\ \Omega(k^{2-\frac{2}{t+2}}) & \mbox{if } t\mbox{ is even.} \end{cases} \]
\end{theorem}

\begin{proof}
First we prove the upper bounds.  We will make use of Theorem~3.23(2) in~\cite{cox}. This theorem implies that if $H$ is neither a matching nor a star and $\ex(n,H) = \Omega(n^\beta)$ for some real number $\beta$, then $\itn(k,H) = O(k^{3-\beta})$.  In the case where $H=C_{2t}$, it was proved by Lazebnik, Ustimenko and Woldar~\cite{c2klower} that $\ex(n,C_{2t}) = \Omega(n^{1+\frac{2}{3t-2}})$ if $t$ is even, and $\ex(n,C_{2t}) = \Omega(n^{1+\frac{2}{3t-3}})$ if $t$ is odd.  Thus we have that $\itn(k,C_{2t}) = O(k^{2-\frac{2}{3t-2}})$ if $t$ is even, and $\itn(k,C_{2t}) = O(k^{2-\frac{2}{3t-3}})$ if $t$ is odd. 

We will give a construction for the lower bound on $\itn(k,C_{2t})$ which is similar to the one used for $C_4$. If $t$ is odd, take a complete bipartite graph $G$ with color classes $A$ and $B$, with sizes $\alpha k^{1-\frac{2}{t+1}}$ and $\beta k$, respectively.  If $t$ is even, take the complete bipartite graph $G$ with color classes $C$ and $D$ of size $\gamma k^{1-\frac{2}{t+2}}$ and $\delta k$, respectively. The constants are chosen small enough so that $(\alpha \beta)^{\frac{t+1}{2t}} + \beta < \frac{1}{2t-3}$ for odd $t \ge 3$, and $\gamma^{\frac{t+2}{2t}}\delta^{\frac{1}{2}} + \delta < \frac{1}{2t-3}$ for even $t \ge 2$.  Then a direct application of Theorem~\ref{Cycle_Theorem} yields that any $C_{2t}$-free subgraph of $G$ has less than $k$ edges when $k$ is sufficiently large.
\end{proof}


\section{Remarks and open questions}\label{conjectures}




We pose two conjectures about the inverse Tur\'an number of the path depending on the parity of its length.  In agreement with the intuition of Briggs and Cox~\cite{cox}, we believe that the inverse Tur\'an number of a path with odd length is attained by a clique.  On the other hand,  we believe that a balanced multipartite graph of $t$ parts is optimal for forcing a path of length $2t$.

\begin{conjecture}
The inverse Tur\'an number of a path of length $2t+1$ is attained asymptotically by a complete graph. Therefore for every $t$,
\[\ex^{-1}(k,P_{2t+1}) = \binom{\floor{\frac{k}{t}}}{2} +o(k^2).\]
\end{conjecture}

\begin{conjecture}
The inverse Tur\'an number of a path of length $2t$ is attained asymptotically by a balanced, complete $t$-partite graph. Therefore for every $t$,
\[\ex^{-1}(k,P_{2t}) =
\frac{k^2}{2(t-1)^2}
\left(1-\frac{1}{t}\right) + o(k^2).\]
\end{conjecture}

We have given upper and lower bounds for the value of  $\ex^{-1}(k,C_4)$, and we conjecture that the lower bound is asymptotically sharp.
\begin{conjecture}
\[\ex^{-1}(k,C_4) = \frac{2\sqrt{2}k^{3/2}}{3\sqrt{3}}+o(k^{3/2}).\]
\end{conjecture}

\section*{Acknowledgments}
We thank Chris Cox for providing us with the reference~\cite{Foucaud} and for discussions on the topic of this paper. We also thank the anonymous referees for their numerous and insightful suggestions. The research of first, second and forth authors was supported by the National Research, Development and Innovation Office NKFIH, grants  K116769, K117879 and K126853. The research of the second author was partially supported by the Shota Rustaveli National Science Foundation of Georgia SRNSFG, grant number  FR-18-2499. The research of the third author was supported by the Institute for Basic Science (IBS-R029-C1).

\end{document}